\documentclass[12pt,twoside]{amsart}
\usepackage{amsmath, amsthm, amscd, amsfonts, amssymb, graphicx}
\usepackage{enumerate}
\usepackage[colorlinks=true,
linkcolor=blue,
urlcolor=cyan,
citecolor=red]{hyperref}
\usepackage{mathrsfs}
\newtheorem{Theorem}{Theorem}[section]
\newtheorem{Lemma}{Lemma}[section]
\newtheorem{Remark}{Remark}[section]

\newtheorem{Corollary}{Corollary}[section]
\newtheorem{example}{Example}[section]

\numberwithin{equation}{section}

\begin{document}
	
\title{Weighted Inequalities For The Numerical Radius}
\author{Shiva Sheybani$^{1}$, Mohammed Sababheh$^{2}$, Hamid Reza Moradi$^{3}$}

\begin{abstract}
In this article, we obtain several new weighted bounds for the numerical radius of a Hilbert space operator.
	The significance of the obtained results is the way they generalize many existing results in the literature; where certain values of the weights imply some known results, or refinements of these results. In the end, we present some numerical examples that show how our results refine the well known results in the literature, related to this topic.
\end{abstract}
\maketitle
%------------------------------------------------------------------------------------%
\pagestyle{myheadings}
\markboth{\centerline {}}
{\centerline {}}
\bigskip
\bigskip
%------------------------------------------------------------------------------------%
%------------------------------------------------------------------------------------%
\section{Introduction} \label{sec:intro}

Let $\mathscr{B}\left( \mathscr{H}\right)$ denote the $C^*$-algebra of all bounded linear operators on a complex Hilbert space $\mathscr{H}.$
In the sequel, upper case letters will be used to denote elements of $\mathscr{B}\left( \mathscr{H}\right)$, while lower case letters will denote real numbers.
For $A\in \mathscr{B}\left( \mathscr{H}\right)$, the numerical radius and operator norm are defined, respectively, by
$$\omega(A)=\sup_{\|x\|=1}\left|\left<Ax,x\right>\right|\;{\text{and}}\|A\|=\sup_{\|x\|=1}\|Ax\|;$$ where $\|\cdot\|$ is the norm induced by the inner product $\left<\cdot,\cdot\right>$ on $\mathscr{H}.$
It is well known that both quantities $\omega(\cdot)$ and $\|\cdot\|$ define equivalent norms on $\mathscr{B}\left( \mathscr{H}\right)$ via the inequalities 
\begin{align}\label{ineq_equiv_intro}
\frac{1}{2}\|A\|\leq\omega(A)\leq \|A\|, A\in \mathscr{B}\left( \mathscr{H}\right).
\end{align}
The significance of such bounds lies in finding easier lower and upper bounds for the numerical radius; due to the difficulty in computing the exact value of $\omega(A).$
Thus, sharpening the bounds in \eqref{ineq_equiv_intro} has received a considerable attention in the literature. 
 
Among the most well established interesting results in this direction are the following inequalities due to Kittaneh \cite{7,07}
\begin{equation}\label{ineq_kitt_1}
\omega(A)\leq \frac{1}{2}\|\;|A|+|A^*|\;\|,
\end{equation}
\begin{equation}\label{ineq_kitt_2}
\omega(A)^2\leq \frac{1}{2}\|\;|A|^2+|A^*|^2\|,
\end{equation}
and
\begin{equation}\label{ineq_kitt_3}
\omega(A)\leq \frac{1}{2}\left(\|A\|+\|A^2\|^{\frac{1}{2}}\right),
\end{equation}
where $A^*$ is the adjoint operator of $A$ and $|A|=(A^*A)^{1/2}.$
We emphasize that the notation $\omega(A)^2$ means $(\omega(A))^2.$
We refer the reader to \cite{10,moradi,8} for some treatments of the inequalities \eqref{ineq_kitt_1}, \eqref{ineq_kitt_2} and \eqref{ineq_kitt_3}. In \cite{sababheh}, a refinement of \eqref{ineq_kitt_2} was shown as follows
\begin{equation}\label{ineq_int}
{{\omega }}\left( A \right)^{2}\le \left\| \int\limits_{0}^{1}{{{\left( \left( 1-t \right)\left| A \right|+t\left| {{A}^{*}} \right| \right)}^{2}}dt} \right\|.
\end{equation}

The Aluthge transform $\widetilde{A}$ of $A\in \mathscr{B}\left( \mathscr{H} \right)$ has appeared in many studies treating numerical radius inequalities, and has provided some interesting refinements.
Recall that the Aluthge transform $\widetilde{A}$ of $A\in \mathscr{B}\left( \mathscr{H} \right)$ is defined by $\widetilde{A}=|A|^{\frac{1}{2}}U|A|^{\frac{1}{2}},$ where $U$ is  the partial isometry appearing in the polar decomposition $A=U|A|$ of $A$, \cite{alu}.
Yamazaki showed the following better estimates of \eqref{ineq_kitt_3} in \cite{yamazaki}
\begin{equation}\label{ineq_yama_1}
\omega \left( A \right)\le \frac{1}{2}\left( \left\| A \right\|+\omega \left( \widetilde{A} \right) \right).
\end{equation}

In this article, we further explore related numerical radius inequalities, by providing weighted versions of a parameter $t$ which, upon selecting certain values, implies the above inequalities or some sharper bounds.

The following results will be needed in our analysis.
\begin{Lemma}\label{11}
	\cite{kato} Let $A\in \mathscr{B}\left( \mathscr{H} \right)$ and let $x,y\in\mathscr{H}$ be any vectors. If $0\le t\le 1$,
	\[{{\left| \left\langle Ax,y \right\rangle  \right|}^{2}}\le \left\langle {{\left| A \right|}^{2\left( 1-t \right)}}x,x \right\rangle \left\langle {{\left| {{A}^{*}} \right|}^{2t}}y,y \right\rangle.\]
\end{Lemma}

\begin{Lemma}\label{12}\cite{mcarthy}
	Let $A\in \mathscr{B}\left( \mathscr{H} \right)$ be a positive operator and let $x\in\mathscr{H}$ be a unit vector. Then
	\[{{\left\langle Ax,x \right\rangle }^{r}}\le \left\langle {{A}^{r}}x,x \right\rangle ,\text{ }\left( r\ge 1 \right),\]
	\[\left\langle {{A}^{r}}x,x \right\rangle \le {{\left\langle Ax,x \right\rangle }^{r}},\text{ }\left( 0\le r\le 1 \right).\]
\end{Lemma}

\begin{Lemma}\label{lemma_aluthge_1} \cite{yamazaki}
	Let $A\in \mathscr{B}\left( \mathscr{H} \right)$. Then
	\[\omega \left( \widetilde{A} \right)\le \left\| \widetilde{A} \right\|\le {{\left\| {{A}^{2}} \right\|}^{\frac{1}{2}}}.\]
\end{Lemma}

\begin{Lemma}\label{lemma_powers}\cite{halmos}
	Let $A\in \mathscr{B}\left( \mathscr{H} \right)$ and let $k\in\mathbb{N}$. Then $$\omega\left(A^k\right)\leq \omega(A)^k.$$
\end{Lemma}

\begin{Lemma}\label{lemma_theta} \cite {yamazaki}
	Let $A\in \mathscr{B}\left( \mathscr{H} \right)$. Then
	$$\omega(A)=\sup_{\theta}\left\|\Re\left(e^{i\theta}A\right)\right\|,$$ where $\Re(T)$ is the real part of the operator $T$, defined by $\Re T=\frac{T+T^*}{2}.$
\end{Lemma}

\begin{Lemma}\cite{ak}\label{lem_amer}
	Let $A,B,C,D\in \mathscr{B}\left( \mathscr{H} \right)$. Then
	$$r(AB+CD)\leq \frac{1}{2}\left( \omega(BA)+\omega(DC)+\sqrt{\left(\omega(BA)-\omega(DC)\right)^2+4\|BC\|\cdot\|DA\| }            \right),$$
	where $r(\cdot)$ is the spectral radius.
\end{Lemma}

\begin{Lemma}\cite{bhatia}\label{lem_log-conv}
	Let $A,B\in \mathscr{B}\left( \mathscr{H} \right)$ be positive. Then
	$$\|A^tB^t\|\leq\|AB\|^t, 0\leq t\leq 1.$$ Further, the function $f(t)=\|A^tB^t\|$ is log-convex on $[0,1].$
\end{Lemma}

%-------------------------------------------------------------------------------------------------------------

\section{Main Results}

\subsection{Results involving the Aluthge transform}
We begin with the following inequality, which gives a weighted upper bound in terms of the Aluthge transform. The value of this result can be seen in 
Corollary \ref{c1} below, where it turns out that this form implies refinements of \eqref{ineq_yama_1}. In the sequel, the notation $\widetilde{A}_{t}$ will be used to denote the weighted Aluthge transform defined by 
$$\widetilde{A}_{t}=|A|^{1-t}U|A|^t, 0\leq t\leq 1,$$ where $U$ is the partial isometry in the polar decomposition $A=U|A|$ of $A$.

\begin{Theorem}\label{t1}
	Let $A\in \mathscr{B}\left( \mathscr{H} \right)$ and $0\leq t\leq 1$. Then
	\begin{small}
		\[\begin{aligned}
		& \omega \left( A \right) \\ 
		& \le \frac{1}{2}\sqrt{\frac{1}{4}\left\| {{\left| A \right|}^{4t}}+{{\left| A \right|}^{4\left( 1-t \right)}} \right\|+\frac{1}{2}{{\left\| A \right\|}^{2}}+\frac{1}{4}\left\| {{\left| {{\widetilde{A}}_{t}} \right|}^{2}}+{{\left| {{\widetilde{A}}_{t}}^{*} \right|}^{2}} \right\|+\frac{1}{2}\omega \left( {{\widetilde{A}}_{t}}^{2} \right)+\left\| {{\left| A \right|}^{2t}}+{{\left| A \right|}^{2\left( 1-t \right)}} \right\|\omega \left( {{\widetilde{A}}_{t}} \right)}. 
		\end{aligned}\]
	\end{small}
\end{Theorem}

\begin{proof}
	Let $A=U|A|$ be the polar decomposition of $A$. Noting the identity
	$$\Re\left<x,y\right>=\frac{1}{4}\left(\|x+y\|^2-\|x-y\|^2\right), x,y\in\mathscr{H},$$ we have
	\[\begin{aligned}
	\operatorname{\Re}\left\langle {{e}^{i\theta }}Ax,x \right\rangle &=\operatorname{\Re}\left\langle {{e}^{i\theta }}U\left| A \right|x,x \right\rangle  \\ 
	& =\operatorname{\Re}\left\langle {{e}^{i\theta }}U{{\left| A \right|}^{t}}{{\left| A \right|}^{1-t}}x,x \right\rangle  \\ 
	& =\operatorname{\Re}\left\langle {{e}^{i\theta }}{{\left| A \right|}^{1-t}}x,{{\left| A \right|}^{t}}{{U}^{*}}x \right\rangle  \\ 
	& =\frac{1}{4}{{\left\| \left( {{e}^{i\theta }}{{\left| A \right|}^{1-t}}+{{\left| A \right|}^{t}}{{U}^{*}} \right)x \right\|}^{2}}-\frac{1}{4}{{\left\| \left( {{e}^{i\theta }}{{\left| A \right|}^{1-t}}-{{\left| A \right|}^{t}}{{U}^{*}} \right)x \right\|}^{2}} \\ 
	& \le \frac{1}{4}{{\left\| \left( {{e}^{i\theta }}{{\left| A \right|}^{1-t}}+{{\left| A \right|}^{t}}{{U}^{*}} \right)x \right\|}^{2}} \\ 
	& \le \frac{1}{4}{{\left\| {{e}^{i\theta }}{{\left| A \right|}^{1-t}}+{{\left| A \right|}^{t}}{{U}^{*}} \right\|}^{2}}.  
	\end{aligned}\]
	On the other hand,
	\[\begin{aligned}
	& \frac{1}{4}{{\left\| {{e}^{i\theta }}{{\left| A \right|}^{1-t}}+{{\left| A \right|}^{t}}{{U}^{*}} \right\|}^{2}} \\ 
	& =\frac{1}{4}\left\| \left( {{e}^{i\theta }}{{\left| A \right|}^{1-t}}+{{\left| A \right|}^{t}}{{U}^{*}} \right){{\left( {{e}^{i\theta }}{{\left| A \right|}^{1-t}}+{{\left| A \right|}^{t}}{{U}^{*}} \right)}^{*}} \right\| \\ 
	& =\frac{1}{4}\left\| {{\left| A \right|}^{2t}}+{{\left| A \right|}^{2\left( 1-t \right)}}+{{e}^{i\theta }}{{\widetilde{A}}_{t}}+{{e}^{-i\theta }}{{\widetilde{A}}_{t}}^{*} \right\| \\ 
	& =\frac{1}{4}{{\left\| {{\left( {{\left| A \right|}^{2t}}+{{\left| A \right|}^{2\left( 1-t \right)}}+{{e}^{i\theta }}{{\widetilde{A}}_{t}}+{{e}^{-i\theta }}{{\widetilde{A}}_{t}}^{*} \right)}^{2}} \right\|}^{\frac{1}{2}}} \\ 
	& =\frac{1}{4}\left\| {{\left( {{\left| A \right|}^{2t}}+{{\left| A \right|}^{2\left( 1-t \right)}} \right)}^{2}}+{{\left( {{e}^{i\theta }}{{\widetilde{A}}_{t}}+{{e}^{-i\theta }}{{\widetilde{A}}_{t}}^{*} \right)}^{2}} \right. \\ 
	&\qquad {{\left. +\left( {{\left| A \right|}^{2t}}+{{\left| A \right|}^{2\left( 1-t \right)}} \right)\left( {{e}^{i\theta }}{{\widetilde{A}}_{t}}+{{e}^{-i\theta }}{{\widetilde{A}}_{t}}^{*} \right)+\left( {{e}^{i\theta }}{{\widetilde{A}}_{t}}+{{e}^{-i\theta }}{{\widetilde{A}}_{t}}^{*} \right)\left( {{\left| A \right|}^{2t}}+{{\left| A \right|}^{2\left( 1-t \right)}} \right) \right\|}^{\frac{1}{2}}} \\ 
	& \le \frac{1}{4}\left( {{\left\| {{\left| A \right|}^{2t}}+{{\left| A \right|}^{2\left( 1-t \right)}} \right\|}^{2}}+{{\left\| {{e}^{i\theta }}{{\widetilde{A}}_{t}}+{{e}^{-i\theta }}{{\widetilde{A}}_{t}}^{*} \right\|}^{2}} \right. \\ 
	&\qquad {{\left. +\left\| \left( {{\left| A \right|}^{2t}}+{{\left| A \right|}^{2\left( 1-t \right)}} \right)\left( {{e}^{i\theta }}{{\widetilde{A}}_{t}}+{{e}^{-i\theta }}{{\widetilde{A}}_{t}}^{*} \right) \right\|+\left\| \left( {{e}^{i\theta }}{{\widetilde{A}}_{t}}+{{e}^{-i\theta }}{{\widetilde{A}}_{t}}^{*} \right)\left( {{\left| A \right|}^{2t}}+{{\left| A \right|}^{2\left( 1-t \right)}} \right) \right\| \right)}^{\frac{1}{2}}} \\ 
	& = \frac{1}{4}\left( \left\| {{\left| A \right|}^{4t}}+{{\left| A \right|}^{4\left( 1-t \right)}}+2{{\left| A \right|}^{2}} \right\|+\left\| {{\left| {{\widetilde{A}}_{t}} \right|}^{2}}+{{\left| {{\widetilde{A}}_{t}}^{*} \right|}^{2}}+2\operatorname{Re}\left( {{e}^{2i\theta }}{{\widetilde{A}}_{t}}^{2} \right) \right\| \right. \\ 
	&\qquad {{\left. +\left\| \left( {{\left| A \right|}^{2t}}+{{\left| A \right|}^{2\left( 1-t \right)}} \right)\left( {{e}^{i\theta }}{{\widetilde{A}}_{t}}+{{e}^{-i\theta }}{{\widetilde{A}}_{t}}^{*} \right) \right\|+\left\| \left( {{e}^{i\theta }}{{\widetilde{A}}_{t}}+{{e}^{-i\theta }}{{\widetilde{A}}_{t}}^{*} \right)\left( {{\left| A \right|}^{2t}}+{{\left| A \right|}^{2\left( 1-t \right)}} \right) \right\| \right)}^{\frac{1}{2}}} \\ 
	& \le \frac{1}{4}\left( \left\| {{\left| A \right|}^{4t}}+{{\left| A \right|}^{4\left( 1-t \right)}}+2{{\left| A \right|}^{2}} \right\|+\left\| {{\left| {{\widetilde{A}}_{t}} \right|}^{2}}+{{\left| {{\widetilde{A}}_{t}}^{*} \right|}^{2}}+2\operatorname{Re}\left( {{e}^{2i\theta }}{{\widetilde{A}}_{t}}^{2} \right) \right\| \right. \\ 
	&\qquad {{\left. +4\left\| {{\left| A \right|}^{2t}}+{{\left| A \right|}^{2\left( 1-t \right)}} \right\|\left\| \operatorname{Re}\left( {{e}^{i\theta }}{{\widetilde{A}}_{t}} \right) \right\| \right)}^{\frac{1}{2}}} \\ 
	& \le \frac{1}{4}\left( \left\| {{\left| A \right|}^{4t}}+{{\left| A \right|}^{4\left( 1-t \right)}} \right\|+2{{\left\| A \right\|}^{2}}+\left\| {{\left| {{\widetilde{A}}_{t}} \right|}^{2}}+{{\left| {{\widetilde{A}}_{t}}^{*} \right|}^{2}} \right\|+2\left\| \operatorname{Re}\left( {{e}^{2i\theta }}{{\widetilde{A}}_{t}}^{2} \right) \right\| \right. \\ 
	&\qquad {{\left. +4\left\| {{\left| A \right|}^{2t}}+{{\left| A \right|}^{2\left( 1-t \right)}} \right\|\left\| \operatorname{Re}\left( {{e}^{i\theta }}{{\widetilde{A}}_{t}} \right) \right\| \right)}^{\frac{1}{2}}} \\ 
	& =\frac{1}{2}\left( \frac{1}{4}\left\| {{\left| A \right|}^{4t}}+{{\left| A \right|}^{4\left( 1-t \right)}} \right\|+\frac{1}{2}{{\left\| A \right\|}^{2}}+\frac{1}{4}\left\| {{\left| {{\widetilde{A}}_{t}} \right|}^{2}}+{{\left| {{\widetilde{A}}_{t}}^{*} \right|}^{2}} \right\|+\frac{1}{2}\left\| \operatorname{Re}\left( {{e}^{2i\theta }}{{\widetilde{A}}_{t}}^{2} \right) \right\| \right. \\ 
	&\qquad {{\left. +\left\| {{\left| A \right|}^{2t}}+{{\left| A \right|}^{2\left( 1-t \right)}} \right\|\left\| \operatorname{Re}\left( {{e}^{i\theta }}{{\widetilde{A}}_{t}} \right) \right\| \right)}^{\frac{1}{2}}} \\ 
	& \le \frac{1}{2}\left( \frac{1}{4}\left\| {{\left| A \right|}^{4t}}+{{\left| A \right|}^{4\left( 1-t \right)}} \right\|+\frac{1}{2}{{\left\| A \right\|}^{2}}+\frac{1}{4}\left\| {{\left| {{\widetilde{A}}_{t}} \right|}^{2}}+{{\left| {{\widetilde{A}}_{t}}^{*} \right|}^{2}} \right\|+\frac{1}{2}\omega \left( {{\widetilde{A}}_{t}}^{2} \right) \right. \\ 
	&\qquad {{\left. +\left\| {{\left| A \right|}^{2t}}+{{\left| A \right|}^{2\left( 1-t \right)}} \right\|\omega \left( {{\widetilde{A}}_{t}} \right) \right)}^{\frac{1}{2}}}.  
	\end{aligned}\]
	Thus,
	\begin{small}
		\[\begin{aligned}
		& \omega \left( A \right) \\ 
		& \le \frac{1}{2}\sqrt{\frac{1}{4}\left\| {{\left| A \right|}^{4t}}+{{\left| A \right|}^{4\left( 1-t \right)}} \right\|+\frac{1}{2}{{\left\| A \right\|}^{2}}+\frac{1}{4}\left\| {{\left| {{\widetilde{A}}_{t}} \right|}^{2}}+{{\left| {{\widetilde{A}}_{t}}^{*} \right|}^{2}} \right\|+\frac{1}{2}\omega \left( {{\widetilde{A}}_{t}}^{2} \right)+\left\| {{\left| A \right|}^{2t}}+{{\left| A \right|}^{2\left( 1-t \right)}} \right\|\omega \left( {{\widetilde{A}}_{t}} \right)}. 
		\end{aligned}\]
	\end{small}
	This completes the proof of the theorem.
\end{proof}

Letting $t=\frac{1}{2}$ in Theorem \ref{t1}, we reach the following result, whose significance is explained next.
\begin{Corollary}\label{c1}
	Let $A\in \mathscr{B}\left( \mathscr{H} \right)$. Then
	\[\omega \left( A \right)\le \frac{1}{2}\sqrt{{{\left\| A \right\|}^{2}}+\frac{1}{4}\left\| {{\left| \widetilde{A} \right|}^{2}}+{{\left| {{\widetilde{A}}^{*}} \right|}^{2}} \right\|+\frac{1}{2}\omega \left( {{\widetilde{A}}^{2}} \right)+2\left\| A \right\|\omega \left( \widetilde{A} \right)}.\]
\end{Corollary}

The significance of Corollary \ref{c1} is shown in the next remark, where multiple refinements of \eqref{ineq_kitt_3} are found. 
\begin{Remark}
	Taking Lemmas \ref{lemma_aluthge_1} and \ref{lemma_powers} into account, it follows from Corollary \ref{c1} that
	\[\begin{aligned}
	\omega \left( A \right)&\le \frac{1}{2}\sqrt{{{\left\| A \right\|}^{2}}+\frac{1}{4}\left\| {{\left| \widetilde{A} \right|}^{2}}+{{\left| {{\widetilde{A}}^{*}} \right|}^{2}} \right\|+\frac{1}{2}\omega \left( {{\widetilde{A}}^{2}} \right)+2\left\| A \right\|\omega \left( \widetilde{A} \right)} \\ 
	& \le \frac{1}{2}\sqrt{{{\left\| A \right\|}^{2}}+\frac{1}{2}{{\left\| \widetilde{A} \right\|}^{2}}+\frac{1}{2}\omega \left( {{\widetilde{A}}^{2}} \right)+2\left\| A \right\|\omega \left( \widetilde{A} \right)} \\ 
	& \le \frac{1}{2}\sqrt{{{\left\| A \right\|}^{2}}+\frac{1}{2}{{\left\| \widetilde{A} \right\|}^{2}}+\frac{1}{2}{{\omega }^{2}}\left( \widetilde{A} \right)+2\left\| A \right\|\omega \left( \widetilde{A} \right)} \\ 
	& \le \frac{1}{2}\sqrt{{{\left\| A \right\|}^{2}}+{{\left\| \widetilde{A} \right\|}^{2}}+2\left\| A \right\|\left\| \widetilde{A} \right\|} \\ 
	&=\frac{1}{2}\sqrt{\left(\|A\|+\left\|\widetilde{A}\right\|\right)^2}\\
	&=\frac{1}{2}\left(\|A\|+\left\|\widetilde{A}\right\|\right)\\
	&\le \frac{1}{2}\left( \left\| A \right\|+{{\left\| {{A}^{2}} \right\|}^{\frac{1}{2}}} \right).
	\end{aligned}\]
\end{Remark}

%-------------------------------------------------------------------------------------------------------------

\subsection{Other weighted forms}

Now we move to another type of weighted versions. The following presents the general form of \eqref{ineq_kitt_2}, which gives \eqref{ineq_kitt_2} upon letting $t=\frac{1}{2}.$
\begin{Theorem}\label{thm_1}
	Let $A\in \mathscr{B}\left( \mathscr{H} \right)$. Then
	\[{{\omega }}\left( A \right)^{2}\le \underset{0\le t\le 1}{\mathop{\min }}\,\left\| \left( 1-t \right){{\left| A \right|}^{\frac{1}{1-t}}}+t{{\left| {{A}^{*}} \right|}^{\frac{1}{t}}} \right\|.\]
\end{Theorem}
\begin{proof}
	Let $x\in \mathscr{H}$ be a unit vector. Then 
	\[\begin{aligned}
	{{\left| \left\langle Ax,x \right\rangle  \right|}^{2}}&\le \left\langle \left| A \right|x,x \right\rangle \left\langle \left| {{A}^{*}} \right|x,x \right\rangle  \\ 
	& =\left\langle {{\left| A \right|}^{\frac{1-t}{1-t}}}x,x \right\rangle \left\langle {{\left| {{A}^{*}} \right|}^{\frac{t}{t}}}x,x \right\rangle  \\ 
	& \le {{\left\langle {{\left| A \right|}^{\frac{1}{1-t}}}x,x \right\rangle }^{1-t}}{{\left\langle {{\left| {{A}^{*}} \right|}^{\frac{1}{t}}}x,x \right\rangle }^{t}} \\ 
	& \le \left( 1-t \right)\left\langle {{\left| A \right|}^{\frac{1}{1-t}}}x,x \right\rangle +t\left\langle {{\left| {{A}^{*}} \right|}^{\frac{1}{t}}}x,x \right\rangle  \\ 
	& =\left\langle \left( \left( 1-t \right){{\left| A \right|}^{\frac{1}{1-t}}}+t{{\left| {{A}^{*}} \right|}^{\frac{1}{t}}} \right)x,x \right\rangle  \\ 
	& \le \left\| \left( 1-t \right){{\left| A \right|}^{\frac{1}{1-t}}}+t{{\left| {{A}^{*}} \right|}^{\frac{1}{t}}} \right\|, 
	\end{aligned}\]
	where the first inequality follows from Lemma \ref{1} using $y=x$ and $t=\frac{1}{2}$, while the second inequality is obtained using the second inequality in Lemma \ref{12} and the third inequality is obtained by the arithmetic-geometric mean inequality.
	Thus, we have shown
	\[{{\left| \left\langle Ax,x \right\rangle  \right|}^{2}}\le \left\| \left( 1-t \right){{\left| A \right|}^{\frac{1}{1-t}}}+t{{\left| {{A}^{*}} \right|}^{\frac{1}{t}}} \right\|.\]
	By taking supremum over all unit vectors $x\in \mathscr{H}$, we get
	\[{{\omega }}\left( A \right)^{2}\le \left\| \left( 1-t \right){{\left| A \right|}^{\frac{1}{1-t}}}+t{{\left| {{A}^{*}} \right|}^{\frac{1}{t}}} \right\|.\]
Now, by taking infimum over all $t\in \left[ 0,1 \right]$, we infer that
	\[{{\omega }}\left( A \right)^{2}\le \underset{0\le t\le 1}{\mathop{\min }}\,\left\| \left( 1-t \right){{\left| A \right|}^{\frac{1}{1-t}}}+t{{\left| {{A}^{*}} \right|}^{\frac{1}{t}}} \right\|,\]
	which completes the proof.
\end{proof}

On the other hand, a weighted version of \eqref{ineq_kitt_1} can be stated as follows. Although the form is different from \eqref{ineq_kitt_1}, we will show in Remark \ref{rem_ned} how \eqref{000} below implies \eqref{ineq_kitt_1}. 

\begin{Theorem}\label{thm_R}
	Let $A\in \mathscr{B}\left( \mathscr{H} \right)$ and let $0\le t\le 1$. Then
	\begin{equation}\label{000}
	{{\omega }}\left( A \right)^{2}\le \frac{1}{2}\left\| {{\left| A \right|}^{2}}+{{\left| {{A}^{*}} \right|}^{2}}-\frac{t\left( 1-t \right)}{R}{{\left( {{\left| A \right|}}-\left| {{A}^{*}} \right| \right)}^{2}} \right\|,
	\end{equation}
	where $R=\max \left\{ t,1-t \right\}$.
\end{Theorem}

\begin{proof}
	By Lemma 3.12 in \cite{sab_laa},
	\begin{equation}\label{03}
	\left( 1-t \right){{\left| A \right|}^{2}}+t{{\left| {{A}^{*}} \right|}^{2}}\le {{\left( \left( 1-t \right)\left| A \right|+t\left| {{A}^{*}} \right| \right)}^{2}}+2R\left( \frac{{{\left| A \right|}^{2}}+{{\left| {{A}^{*}} \right|}^{2}}}{2}-{{\left( \frac{\left| A \right|+\left| {{A}^{*}} \right|}{2} \right)}^{2}} \right),
	\end{equation}
	where $0\le t\le 1$ and $R=\max \left\{ t,1-t \right\}$. 
	
	On the other hand,
	\[\begin{aligned}
	& \left( 1-t \right){{\left| A \right|}^{2}}+t{{\left| {{A}^{*}} \right|}^{2}}-{{\left( \left( 1-t \right)\left| A \right|+t\left| {{A}^{*}} \right| \right)}^{2}}-{{\left( 1-t \right)}^{2}}{{\left( \left| A \right|-\left| {{A}^{*}} \right| \right)}^{2}} \\ 
	& =\left( 1-t \right){{\left| A \right|}^{2}}+t{{\left| {{A}^{*}} \right|}^{2}}-\left( {{\left( 1-t \right)}^{2}}{{\left| A \right|}^{2}}+{{t}^{2}}{{\left| {{A}^{*}} \right|}^{2}}+t\left( 1-t \right)\left( \left| A \right|\left| {{A}^{*}} \right|+\left| {{A}^{*}} \right|\left| A \right| \right) \right) \\ 
	& -{{\left( 1-t \right)}^{2}}\left( {{\left| A \right|}^{2}}+{{\left| {{A}^{*}} \right|}^{2}}-\left( \left| A \right|\left| {{A}^{*}} \right|+\left| {{A}^{*}} \right|\left| A \right| \right) \right) \\ 
	& =\left( 1-t \right)\left( 2t-1 \right)\left( {{\left| A \right|}^{2}}+{{\left| {{A}^{*}} \right|}^{2}}-\left| A \right|\left| {{A}^{*}} \right|-\left| {{A}^{*}} \right|\left| A \right| \right) \\ 
	& =\left( 1-t \right)\left( 2t-1 \right){{\left( \left| A \right|-\left| {{A}^{*}} \right| \right)}^{2}}.  
	\end{aligned}\]
	Hence,
	\[\left( 1-t \right){{\left| A \right|}^{2}}+t{{\left| {{A}^{*}} \right|}^{2}}-{{\left( \left( 1-t \right)\left| A \right|+t\left| {{A}^{*}} \right| \right)}^{2}}-{{\left( 1-t \right)}^{2}}{{\left( \left| A \right|-\left| {{A}^{*}} \right| \right)}^{2}}=\left( 1-t \right)\left( 2t-1 \right){{\left( \left| A \right|-\left| {{A}^{*}} \right| \right)}^{2}}.\]
	Consequently,
	\begin{equation}\label{04}
	\left( 1-t \right){{\left| A \right|}^{2}}+t{{\left| {{A}^{*}} \right|}^{2}}-{{\left( \left( 1-t \right)\left| A \right|+t\left| {{A}^{*}} \right| \right)}^{2}}=t\left( 1-t \right){{\left( \left| A \right|-\left| {{A}^{*}} \right| \right)}^{2}}.
	\end{equation}
	By \eqref{04} and \eqref{03}, we have
	\begin{equation}\label{1}
	{{\left( \frac{\left| A \right|+\left| {{A}^{*}} \right|}{2} \right)}^{2}}\le \frac{{{\left| A \right|}^{2}}+{{\left| {{A}^{*}} \right|}^{2}}}{2}-\frac{t\left( 1-t \right)}{2R}{{\left( \left| A \right|-\left| {{A}^{*}} \right| \right)}^{2}}.
	\end{equation}
	On the other hand,
	\begin{equation}\label{05}
	\begin{aligned}
	{{\left| \left\langle Ax,x \right\rangle  \right|}^{2}}&\le \left\langle \left| A \right|x,x \right\rangle \left\langle \left| {{A}^{*}} \right|x,x \right\rangle \quad \text{(by Lemma \ref{11})} \\ 
	& \le {{\left( \frac{\left\langle \left| A \right|x,x \right\rangle +\left\langle \left| {{A}^{*}} \right|x,x \right\rangle }{2} \right)}^{2}} \\ 
	&\qquad \text{(by the arithmetic-geometric mean inequality)}\\
	& ={{\left\langle \left( \frac{\left| A \right|+\left| {{A}^{*}} \right|}{2} \right)x,x \right\rangle }^{2}} \\ 
	& \le \left\langle {{\left( \frac{\left| A \right|+\left| {{A}^{*}} \right|}{2} \right)}^{2}}x,x \right\rangle   \quad \text{(by Lemma \ref{12})}.
	\end{aligned}
	\end{equation}
	Therefore,
	\begin{equation}\label{16}
	{{\left| \left\langle Ax,x \right\rangle  \right|}^{2}}\le \left\langle \left( \frac{{{\left| A \right|}^{2}}+{{\left| {{A}^{*}} \right|}^{2}}}{2}-\frac{t\left( 1-t \right)}{2R}{{\left( {{\left| A \right|}}-\left| {{A}^{*}} \right| \right)}^{2}} \right)x,x \right\rangle .
	\end{equation}
	Taking the supremum in \eqref{16} over $x\in\mathscr{H}$, $\left\| x \right\|=1$ we deduce the desired result.
\end{proof}

\begin{Remark}
	It follows from the inequality \eqref{03} that
	\[\frac{\left| A \right|\left| {{A}^{*}} \right|+\left| {{A}^{*}} \right|\left| A \right|}{2}\le \frac{{{\left| A \right|}^{2}}+{{\left| {{A}^{*}} \right|}^{2}}}{2}-\frac{1}{R}\left( \left( 1-t \right){{\left| A \right|}^{2}}+t{{\left| {{A}^{*}} \right|}^{2}}-{{\left( \left( 1-t \right)\left| A \right|+t\left| {{A}^{*}} \right| \right)}^{2}} \right).\]
\end{Remark}

\begin{Remark}\label{rem_ned}
	Analyzing the inequality \eqref{000}, we see that the inequality is best attained when $\frac{t(1-t)}{R}$  attains its maximum value. This occurs when $t=1/2$ and the maximum is $1/2.$ Substituting these values into \eqref{000} we reach the following
	\[\begin{aligned}
	{{\omega }}\left( A \right)^{2}&\le \left\| \frac{{{\left| A \right|}^{2}}+{{\left| {{A}^{*}} \right|}^{2}}}{2}-\frac{1}{4}{{\left( \left| A \right|-\left| {{A}^{*}} \right| \right)}^{2}} \right\| \\ 
	& =\left\| {{\left( \frac{\left| A \right|+\left| {{A}^{*}} \right|}{2} \right)}^{2}} \right\| \\ 
	& ={{\left\| \frac{\left| A \right|+\left| {{A}^{*}} \right|}{2} \right\|}^{2}}.  
	\end{aligned}\]
	Thus, we have shown that
	$$\omega(A)\leq\frac{1}{2}\|\;|A|+|A^*|\;\|,$$ which is \eqref{ineq_kitt_1}. Thus, Theorem \ref{thm_R} provides a new proof of \eqref{ineq_kitt_1}.
\end{Remark}

Another weighted inequality that refines \eqref{ineq_kitt_3} can be stated as follows.
\begin{Theorem}\label{thm_prod}
	Let $A\in \mathscr{B}\left( \mathscr{H} \right)$. Then
	\begin{align*}
	\omega(A)\leq \frac{1}{2}\left(\|A\|+\sqrt{\|\;|A|^t|A^*|^t\|\cdot\|\;|A|^{1-t}|A^*|^{1-t}\|}\right), 0\leq t\leq 1.
	\end{align*}
\end{Theorem}

\begin{proof}
	Since $|A|+|A^*|$ is a positive operator, it follows that $\|\;|A|+|A^*|\;\|=r(|A|+|A^*|).$ Therefore, Lemma \ref{lem_amer} implies
	\begin{align*}
	\omega(A)&\leq\frac{1}{2}\|\;|A|+|A^*|\;\|\\
	&=\frac{1}{2}r(|A|+|A^*|)\\
	&=\frac{1}{2}r\left(|A|^t|A|^{1-t}+|A^*|^{1-t}|A^*|^{t}\right)\\
	&\leq\frac{1}{4}\left(\omega(|A|)+\omega(|A^*|)+\sqrt{\left(\omega(|A|)-\omega(|A^*|)\right)^2+4\|\;|A|^t|A^*|^t\|\cdot\|\;|A|^{1-t}|A^*|^{1-t}\|}\right)\\
	&=\frac{1}{4}\left(\|\;|A|\;\|+\|\;|A^*|\;\|+2\sqrt{  \|\;|A|^t|A^*|^t|\cdot\|\;|A|^{1-t}|A^*|^{1-t}\| }\right)\\
	&=\frac{1}{2}\left(\|A\|+\sqrt{  \|\;|A|^t|A^*|^t\|\cdot\|\;|A|^{1-t}|A^*|^{1-t}\| }\right).
	\end{align*}
	This completes the proof.
\end{proof}

\begin{Remark}\label{rem_ned1}
	To see how Theorem \ref{thm_prod} refines \eqref{ineq_kitt_3}, we use Lemma \ref{lem_log-conv} to find that
	\begin{align*}
	\sqrt{  \|\;|A|^t|A^*|^t\|\cdot\|\;|A|^{1-t}|A^*|^{1-t}\| }&\leq \sqrt{\|\;|A|\;|A^*|\;\|^t\|\;|A|\;|A^*|\;\|^{1-t}}\\
	&=\sqrt{\|\;|A|\;|A^*|\|}\\
	&=\|A^2\|^{\frac{1}{2}}.
	\end{align*}
\end{Remark}

\begin{Corollary}
	Let $A\in \mathscr{B}\left( \mathscr{H} \right)$ and $0\leq t\leq 1$. Then
	\begin{align*}
	\omega(A)&\leq\frac{1}{2}\left(\|A\|+\left\|\;|A|^{\frac{1}{2}}|A^*|^{\frac{1}{2}}\right\|\right)\\
	&\leq \frac{1}{2}\left(\|A\|+\sqrt{ \|\;|A|^{t}|A^*|^{t}\|\cdot  \|\;|A|^{1-t}|A^*|^{1-t}\|      }\right)\\
	&\leq\frac{1}{2}\left(\|A\|+\|A^2\|^{\frac{1}{2}}\right).
	\end{align*}
\end{Corollary}

\begin{proof}
	By Lemma \ref{lem_log-conv}, both 
	$$\|\;|A|^{t}|A^*|^{t}\|\;{\text{and}}\;\|\;|A|^{1-t}|A^*|^{1-t}\|$$ are log-convex functions in $t$, $0\leq t\leq 1$. Consequently, the function
	$$f(t)=\sqrt{ \|\;|A|^{t}|A^*|^{t}\|\cdot  \|\;|A|^{1-t}|A^*|^{1-t}\|      }$$ is log-convex, hence is convex. Since $f$ is convex and symmetric about $t=\frac{1}{2},$ it follows that $f$ attains its minimum at $\frac{1}{2}.$ Thus,
	$$f\left(\frac{1}{2}\right)\leq f(t), 0\leq t\leq 1.$$ This together with Theorem \ref{thm_prod} and Remark \ref{rem_ned1} imply the desired result.
\end{proof}

In the next result, we present a refinement of \eqref{ineq_int}.
\begin{Theorem}
	Let $A\in \mathscr{B}\left( \mathscr{H} \right)$. Then
	\[{{\omega }}\left( A \right)^{2}\le \left\| \int\limits_{0}^{1}{{{\left( \left( 1-t \right)\left| A \right|+t\left| {{A}^{*}} \right| \right)}^{2}}dt}-\frac{1}{48}{{\left( \left| A \right|-\left| {{A}^{*}} \right| \right)}^{2}} \right\|.\]
\end{Theorem}

\begin{proof}
	Obviously,
	\[\begin{aligned}
	& {{\left( \frac{\left| A \right|+\left| {{A}^{*}} \right|}{2} \right)}^{2}} \\ 
	& ={{\left( \frac{\left( 1-t \right)\left| A \right|+t\left| {{A}^{*}} \right|+\left( 1-t \right)\left| {{A}^{*}} \right|+t\left| A \right|}{2} \right)}^{2}} \\ 
	& \le \frac{{{\left( \left( 1-t \right)\left| A \right|+t\left| {{A}^{*}} \right| \right)}^{2}}+{{\left( \left( 1-t \right)\left| {{A}^{*}} \right|+t\left| A \right| \right)}^{2}}}{2}-\frac{t\left( 1-t \right){{\left( 1-2t \right)}^{2}}}{2R}{{\left( \left| A \right|-\left| {{A}^{*}} \right| \right)}^{2}},
	\end{aligned}\]
	where we have used \eqref{03} to obtain the last inequality.
	Since 
	$R=\max \left\{ t,1-t \right\}=\frac{1-t+t+\left| 1-t-t \right|}{2}=\frac{1+\left| 1-2t \right|}{2},$
	we get
	\[\begin{aligned}
	& {{\left( \frac{\left| A \right|+\left| {{A}^{*}} \right|}{2} \right)}^{2}} \\ 
	& \le \frac{{{\left( \left( 1-t \right)\left| A \right|+t\left| {{A}^{*}} \right| \right)}^{2}}+{{\left( \left( 1-t \right)\left| {{A}^{*}} \right|+t\left| A \right| \right)}^{2}}}{2}-\frac{t\left( 1-t \right){{\left( 1-2t \right)}^{2}}}{1+\left| 1-2t \right|}{{\left( \left| A \right|-\left| {{A}^{*}} \right| \right)}^{2}}. \\ 
	\end{aligned}\]
	By taking integral over $0\le t\le 1$, and using the fact that
	\[\int\limits_{0}^{1}{{{\left( \left( 1-t \right)\left| A \right|+t\left| {{A}^{*}} \right| \right)}^{2}}dt}=\int\limits_{0}^{1}{{{\left( \left( 1-t \right)\left| {{A}^{*}} \right|+t\left| A \right| \right)}^{2}}dt},\]
	we can obtain
	\[{{\left( \frac{\left| A \right|+\left| {{A}^{*}} \right|}{2} \right)}^{2}}\le \int\limits_{0}^{1}{{{\left( \left( 1-t \right)\left| A \right|+t\left| {{A}^{*}} \right| \right)}^{2}}dt}-\frac{1}{48}{{\left( \left| A \right|-\left| {{A}^{*}} \right| \right)}^{2}}.\]
	By the same method used in the proof of inequality \eqref{05}, we have
	\[{{\left| \left\langle Ax,x \right\rangle  \right|}^{2}}\le \left\langle {{\left( \frac{\left| A \right|+\left| {{A}^{*}} \right|}{2} \right)}^{2}}x,x \right\rangle .\]
	Thus,
	\[{{\left| \left\langle Ax,x \right\rangle  \right|}^{2}}\le \left\langle \left( \int\limits_{0}^{1}{{{\left( \left( 1-t \right)\left| A \right|+t\left| {{A}^{*}} \right| \right)}^{2}}dt}-\frac{1}{48}{{\left( \left| A \right|-\left| {{A}^{*}} \right| \right)}^{2}} \right)x,x \right\rangle ,\]
	which implies
	\[{{\omega }}\left( A \right)^{2}\le \left\| \int\limits_{0}^{1}{{{\left( \left( 1-t \right)\left| A \right|+t\left| {{A}^{*}} \right| \right)}^{2}}dt}-\frac{1}{48}{{\left( \left| A \right|-\left| {{A}^{*}} \right| \right)}^{2}} \right\|,\]
	completing the proof.
\end{proof}

We conclude this subsection with the following weighted inequality, which implies \eqref{ineq_kitt_2} upon letting $t=\frac{1}{2}.$

\begin{Theorem}\label{thm_w}
	Let $A\in \mathscr{B}(\mathscr{H})$. Then 
	\[{{\omega }}\left( A \right)^{2}\le \underset{0\le t\le 1}{\mathop{\min }}\,\left\| \frac{{{\left| A \right|}^{4\left( 1-t \right)}}+{{\left| {{A}^{*}} \right|}^{4t}}}{4}+\frac{\left( 1-t \right){{\left| A \right|}^{2}}+t{{\left| {{A}^{*}} \right|}^{2}}}{2} \right\|.\]
\end{Theorem}

\begin{proof}
	For any unit vector $x\in\mathscr{H}$ and $0\leq t\leq 1$, we have
	\[\begin{aligned}
	{{\left| \left\langle Ax,x \right\rangle  \right|}^{2}}&\le \left\langle {{\left| A \right|}^{2\left( 1-t \right)}}x,x \right\rangle \left\langle {{\left| {{A}^{*}} \right|}^{2t}}x,x \right\rangle  \\ 
	& \le {{\left( \frac{\left\langle {{\left| A \right|}^{2\left( 1-t \right)}}x,x \right\rangle +\left\langle {{\left| {{A}^{*}} \right|}^{2t}}x,x \right\rangle }{2} \right)}^{2}} \\ 
	& =\frac{{{\left\langle {{\left| A \right|}^{2\left( 1-t \right)}}x,x \right\rangle }^{2}}+{{\left\langle {{\left| {{A}^{*}} \right|}^{2t}}x,x \right\rangle }^{2}}+2\left\langle {{\left| A \right|}^{2\left( 1-t \right)}}x,x \right\rangle \left\langle {{\left| {{A}^{*}} \right|}^{2t}}x,x \right\rangle }{4} \\ 
	& \le \frac{\left\langle {{\left| A \right|}^{4\left( 1-t \right)}}x,x \right\rangle +\left\langle {{\left| {{A}^{*}} \right|}^{4t}}x,x \right\rangle +2\left\langle {{\left| A \right|}^{2\left( 1-t \right)}}x,x \right\rangle \left\langle {{\left| {{A}^{*}} \right|}^{2t}}x,x \right\rangle }{4} \\ 
	& \le \frac{\left\langle {{\left| A \right|}^{4\left( 1-t \right)}}x,x \right\rangle +\left\langle {{\left| {{A}^{*}} \right|}^{4t}}x,x \right\rangle +2{{\left\langle {{\left| A \right|}^{2}}x,x \right\rangle }^{1-t}}{{\left\langle {{\left| {{A}^{*}} \right|}^{2}}x,x \right\rangle }^{t}}}{4} \\ 
	& \le \frac{\left\langle {{\left| A \right|}^{4\left( 1-t \right)}}x,x \right\rangle +\left\langle {{\left| {{A}^{*}} \right|}^{4t}}x,x \right\rangle +2\left( \left( 1-t \right)\left\langle {{\left| A \right|}^{2}}x,x \right\rangle +t\left\langle {{\left| {{A}^{*}} \right|}^{2}}x,x \right\rangle  \right)}{4} \\ 
	& =\left\langle \left( \frac{{{\left| A \right|}^{4\left( 1-t \right)}}+{{\left| {{A}^{*}} \right|}^{4t}}}{4}+\frac{\left( 1-t \right){{\left| A \right|}^{2}}+t{{\left| {{A}^{*}} \right|}^{2}}}{2} \right)x,x \right\rangle  .
	\end{aligned}\]
	Taking the supremum over unit vectors $x$ implies 
	\[{{\omega }}\left( A \right)^{2}\le \left\| \frac{{{\left| A \right|}^{4\left( 1-t \right)}}+{{\left| {{A}^{*}} \right|}^{4t}}}{4}+\frac{\left( 1-t \right){{\left| A \right|}^{2}}+t{{\left| {{A}^{*}} \right|}^{2}}}{2} \right\|.\]
	This completes the proof.
\end{proof}

%---------------------------------------------------------------------------------------------------------------

\section{Further results}

As an important tool to obtain numerical radius inequalities, we present the following inequality for the inner product of Schwarz type. First, notice  Lemmas \ref{11} and \ref{12} imply
\[\begin{aligned}
{{\left| \left\langle Ax,x \right\rangle  \right|}^{2}}&\le \left\langle \left| A \right|x,x \right\rangle \left\langle \left| {{A}^{*}} \right|x,x \right\rangle  \\ 
& =\sqrt{{{\left\langle \left| A \right|x,x \right\rangle }^{2}}{{\left\langle \left| {{A}^{*}} \right|x,x \right\rangle }^{2}}} \\ 
& \le \sqrt{\left\langle {{\left| A \right|}^{2}}x,x \right\rangle \left\langle {{\left| {{A}^{*}} \right|}^{2}}x,x \right\rangle.}  
\end{aligned}\]
So,
\[{{\left| \left\langle Ax,x \right\rangle  \right|}^{2}}\le \sqrt{\left\langle {{\left| A \right|}^{2}}x,x \right\rangle \left\langle {{\left| {{A}^{*}} \right|}^{2}}x,x \right\rangle },\]
for any $A\in\mathscr{B}(\mathscr{H})$ and $x\in\mathscr{H}.$
In the next result, we improve the last inequality, to obtain a form that enables us to find a new weighted numerical radius inequality.
\begin{Theorem}
	Let $A,B\in \mathscr{B}(\mathscr{H})$. Then for any vector $x\in \mathscr{H}$,
\[\left| \left\langle {{B}^{*}}Ax,x \right\rangle -\left\langle {{B}^{*}}x,x \right\rangle \left\langle Ax,x \right\rangle  \right|\le \sqrt{\left\langle {{\left| A \right|}^{2}}x,x \right\rangle \left\langle {{\left| B \right|}^{2}}x,x \right\rangle }-\left| \left\langle Ax,x \right\rangle  \right|\left| \left\langle Bx,x \right\rangle  \right|.\]	
	In particular,
	\begin{equation}\label{020}
{{\left| \left\langle Ax,x \right\rangle  \right|}^{2}}+\left| \left\langle {{A}^{2}}x,x \right\rangle -{{\left\langle Ax,x \right\rangle }^{2}} \right|\le \sqrt{\left\langle {{\left| A \right|}^{2}}x,x \right\rangle \left\langle {{\left| {{A}^{*}} \right|}^{2}}x,x \right\rangle }.	
	\end{equation}
\end{Theorem}

\begin{proof}
	First of all, notice that
	\begin{equation}\label{19}
	\left\| \left( A-\left\langle Ax,x \right\rangle  \right)x \right\|=\sqrt{\left\langle {{\left| A \right|}^{2}}x,x \right\rangle -{{\left| \left\langle Ax,x \right\rangle  \right|}^{2}}}.
	\end{equation}
	Replacing $A$ by $B$ in \eqref{19}, we get
	\[\left\| \left( {B}-\left\langle {B}x,x \right\rangle  \right)x \right\|=\sqrt{\left\langle {{\left| {B} \right|}^{2}}x,x \right\rangle -{{\left| \left\langle Bx,x \right\rangle  \right|}^{2}}}.\]
	On the other hand,	by the Schwarz inequality,
\[\begin{aligned}
   \left| \left\langle {{B}^{*}}Ax,x \right\rangle -\left\langle {{B}^{*}}x,x \right\rangle \left\langle Ax,x \right\rangle  \right|&=\left| \left\langle \left( {{B}^{*}}-\left\langle {{B}^{*}}x,x \right\rangle  \right)\left( A-\left\langle Ax,x \right\rangle  \right)x,x \right\rangle  \right| \\ 
 & =\left| \left\langle A-\left\langle Ax,x \right\rangle x,B-\left\langle Bx,x \right\rangle x \right\rangle  \right| \\ 
 & \le \left\| A-\left\langle Ax,x \right\rangle x \right\|\left\| B-\left\langle Bx,x \right\rangle x \right\|.  
\end{aligned}\]
We can conclude from the discussion above that
\[\begin{aligned}
   \left| \left\langle {{B}^{*}}Ax,x \right\rangle -\left\langle {{B}^{*}}x,x \right\rangle \left\langle Ax,x \right\rangle  \right|&\le \sqrt{\left\langle {{\left| A \right|}^{2}}x,x \right\rangle -{{\left| \left\langle Ax,x \right\rangle  \right|}^{2}}}\sqrt{\left\langle {{\left| B \right|}^{2}}x,x \right\rangle -{{\left| \left\langle Bx,x \right\rangle  \right|}^{2}}} \\ 
 &\le \sqrt{\left\langle {{\left| A \right|}^{2}}x,x \right\rangle \left\langle {{\left| B \right|}^{2}}x,x \right\rangle }-\left| \left\langle Ax,x \right\rangle  \right|\left| \left\langle Bx,x \right\rangle  \right|  
\end{aligned}\]
where the last inequality follows from the simple inequality $\left( {{a}^{2}}-{{b}^{2}} \right)\left( {{c}^{2}}-{{d}^{2}} \right)\le {{\left( ac-bd \right)}^{2}}$ for $a,b,c,d\in \mathbb{R}^+$. 
	Consequently,
	\begin{equation}\label{20}
\left| \left\langle {{B}^{*}}Ax,x \right\rangle -\left\langle {{B}^{*}}x,x \right\rangle \left\langle Ax,x \right\rangle  \right|\le \sqrt{\left\langle {{\left| A \right|}^{2}}x,x \right\rangle \left\langle {{\left| B \right|}^{2}}x,x \right\rangle }-\left| \left\langle Ax,x \right\rangle  \right|\left| \left\langle Bx,x \right\rangle  \right|
	\end{equation}
	which implies the desired inequality.\\
	 Putting ${{B}^{*}}=A$ in \eqref{20}, we obtain the inequality \eqref{020}.
\end{proof}

\begin{Corollary}\label{cor_CS_ref}
	Let $A,B\in \mathscr{B}(\mathscr{H})$. Then for any  vector $x\in \mathscr{H}$,
\[\left| \left\langle {{B}^{*}}x,x \right\rangle \left\langle Ax,x \right\rangle  \right|\le \frac{\sqrt{\left\langle {{\left| A \right|}^{2}}x,x \right\rangle \left\langle {{\left| B \right|}^{2}}x,x \right\rangle }+\left| \left\langle {{B}^{*}}Ax,x \right\rangle  \right|}{2}.\]
\end{Corollary}

\begin{proof}
	Using the triangle inequality for the modulus,
	\begin{equation}\label{21}
\begin{aligned}
   \left| \left\langle Bx,x \right\rangle  \right|\left| \left\langle Ax,x \right\rangle  \right|-\left| \left\langle {{B}^{*}}Ax,x \right\rangle  \right|&=\left| \left\langle {{B}^{*}}x,x \right\rangle \left\langle Ax,x \right\rangle  \right|-\left| \left\langle {{B}^{*}}Ax,x \right\rangle  \right| \\ 
 & \le \left| \left\langle {{B}^{*}}Ax,x \right\rangle -\left\langle {{B}^{*}}x,x \right\rangle \left\langle Ax,x \right\rangle  \right|.  
\end{aligned}
	\end{equation}
	Combining the inequalities \eqref{20} and \eqref{21}, we get
\[\left| \left\langle {{B}^{*}}x,x \right\rangle \left\langle Ax,x \right\rangle  \right|\le \frac{\sqrt{\left\langle {{\left| A \right|}^{2}}x,x \right\rangle \left\langle {{\left| B \right|}^{2}}x,x \right\rangle }+\left| \left\langle {{B}^{*}}Ax,x \right\rangle  \right|}{2}.\]
	This completes the proof.
\end{proof}

\begin{Corollary}
	Let $A\in \mathscr{B}(\mathscr{H})$. Then
	\[\omega(A)^2\leq \frac{1}{2}\min_{0\leq t\leq 1}\left(\left\|t|A|^{\frac{2}{t}}+(1-t)|A^*|^{\frac{2}{1-t}}\right\|^{\frac{1}{2}}+\omega(A^2)\right).\] 
\end{Corollary}

\begin{proof}
	From Corollary \ref{cor_CS_ref}, we have
	\[{{\left| \left\langle Ax,x \right\rangle  \right|}^{2}}\le \frac{\sqrt{\left\langle {{\left| A \right|}^{2}}x,x \right\rangle \left\langle {{\left| {{A}^{*}} \right|}^{2}}x,x \right\rangle }+\left| \left\langle {{A}^{2}}x,x \right\rangle  \right|}{2},\] 
	for any unit vector $x$. Then for $0\leq t\leq 1,$
	\begin{align*}
	\left|\left<Ax,x\right>\right|^2&\leq \frac{\sqrt{\left<|A|^{\frac{2t}{t}}x,x\right>\left<|A^*|^{\frac{2(1-t)}{1-t}}x,x\right>}+|\left<A^2x,x\right>|}{2}\\
	&\leq \frac{\sqrt{\left<|A|^{\frac{2}{t}}x,x\right>^t\left<|A^*|^{\frac{2}{1-t}}x,x\right>^{1-t}}+|\left<A^2x,x\right>|}{2}\\
	&\leq\frac{\sqrt{\left<\left(t|A|^{\frac{2}{t}}+(1-t)|A^*|^{\frac{2}{1-t}}\right)x,x\right>}+|\left<A^2x,x\right>|}{2}.
	\end{align*}
	Taking the supremum over all unit vectors $x$ implies the desired inequality.
\end{proof}

\section{Examples}
In this section, we present different examples to show that the obtained results provide non-trivial refinements of the well known results; such as \eqref{ineq_kitt_1}, \eqref{ineq_kitt_2} and \eqref{ineq_kitt_3}.

\begin{example}
In this example, we investigate the inequality
\[{{\omega }}\left( A \right)^{2}\le \underset{0\le t\le 1}{\mathop{\min }}\,\left\| \left( 1-t \right){{\left| A \right|}^{\frac{1}{1-t}}}+t{{\left| {{A}^{*}} \right|}^{\frac{1}{t}}} \right\|\] obtained in Theorem \ref{thm_1}. 

Let $A=\left[ \begin{matrix}
   0 & 2 & 0  \\
   0 & 0 & 3  \\
   4 & 0 & 0  \\
\end{matrix} \right]$. Then $\left| {{A}^{*}} \right|=\left[ \begin{matrix}
   2 & 0 & 0  \\
   0 & 3 & 0  \\
   0 & 0 & 4  \\
\end{matrix} \right]$  and $\left| A \right|=\left[ \begin{matrix}
   4 & 0 & 0  \\
   0 & 2 & 0  \\
   0 & 0 & 3  \\
\end{matrix} \right]$.
Thus,
	\[\begin{aligned}
  & \underset{0\le t\le 1}{\mathop{\min }}\,\left\| \left( 1-t \right){{\left| A \right|}^{\frac{1}{1-t}}}+t{{\left| {{A}^{*}} \right|}^{\frac{1}{t}}} \right\|\\
  &=\underset{0\le t\le 1}{\mathop{\min }}\,\left\| \left[ \begin{matrix}
   \left( 1-t \right){{4}^{\frac{1}{1-t}}}+t{{2}^{\frac{1}{t}}} & 0 & 0  \\
   0 & \left( 1-t \right){{2}^{\frac{1}{1-t}}}+t{{3}^{\frac{1}{t}}} & 0  \\
   0 & 0 & \left( 1-t \right){{3}^{\frac{1}{1-t}}}+t{{4}^{\frac{1}{t}}}  \\
\end{matrix} \right] \right\| \\ 
 & =\underset{0\le t\le 1}{\mathop{\min }}\,\max \left\{ \left( 1-t \right){{4}^{\frac{1}{1-t}}}+t{{2}^{\frac{1}{t}}},\left( 1-t \right){{2}^{\frac{1}{1-t}}}+t{{3}^{\frac{1}{t}}},\left( 1-t \right){{3}^{\frac{1}{1-t}}}+t{{4}^{\frac{1}{t}}} \right\} \\ 
 & =\underset{0\le t\le 1}{\mathop{\min }}\,\left\{ \begin{aligned}
  & \left( 1-t \right){{3}^{\frac{1}{1-t}}}+t{{4}^{\frac{1}{t}}}\text{ when }0<t\le 0.5557 \\ 
 & \left( 1-t \right){{4}^{\frac{1}{1-t}}}+t{{2}^{\frac{1}{t}}}\text{ when }0.5558\le t<1 \\ 
\end{aligned} \right.\approx 12.002.
\end{aligned}\]

On the other hand,
	\[\frac{1}{2}\left\| {{\left| A \right|}^{2}}+{{\left| {{A}^{*}} \right|}^{2}} \right\|=12.5.\]
Consequently, in this case
	\[\underset{0\le t\le 1}{\mathop{\min }}\,\left\| \left( 1-t \right){{\left| A \right|}^{\frac{1}{1-t}}}+t{{\left| {{A}^{*}} \right|}^{\frac{1}{t}}} \right\|<\frac{1}{2}\left\| {{\left| A \right|}^{2}}+{{\left| {{A}^{*}} \right|}^{2}} \right\|.\]
	Thus, this example shows that Theorem \ref{thm_1} provides a non-trivial refinement of \eqref{ineq_kitt_2}.
	
In fact, this example shows also that Theorem \ref{thm_1} provides a non-trivial refinement of \eqref{ineq_kitt_1}, because Theorem \ref{thm_1} implies

\begin{align*}
\omega(A)\leq \underset{0\le t\le 1}{\mathop{\min }}\,\left\| \left( 1-t \right){{\left| A \right|}^{\frac{1}{1-t}}}+t{{\left| {{A}^{*}} \right|}^{\frac{1}{t}}} \right\|^{\frac{1}{2}}&\approx\sqrt{12.002}
<3.5
=\frac{1}{2}\left\|\;|A|+|A^*|\;\right\|.
\end{align*}
Indeed, the inequality \eqref{ineq_kitt_1} is better than \eqref{ineq_kitt_2} and \eqref{ineq_kitt_3}. Thus, showing that our result is better than \eqref{ineq_kitt_1} implies that it is better than both \eqref{ineq_kitt_2} and \eqref{ineq_kitt_3}.
\end{example}

\begin{example}
In this example, we show that the inequality
\[{{\omega }}\left( A \right)^{2}\le \underset{0\le t\le 1}{\mathop{\min }}\,\left\| \frac{{{\left| A \right|}^{4\left( 1-t \right)}}+{{\left| {{A}^{*}} \right|}^{4t}}}{4}+\frac{\left( 1-t \right){{\left| A \right|}^{2}}+t{{\left| {{A}^{*}} \right|}^{2}}}{2} \right\|\] obtained in Theorem \ref{thm_w} provides a non-trivial refinement of the three inequalities \eqref{ineq_kitt_1}, \eqref{ineq_kitt_2} and \eqref{ineq_kitt_3}. To do so, it suffices to show that it is better than \eqref{ineq_kitt_1}.

Let $A=\left[ \begin{matrix}
   0 & 3 & 0  \\
   0 & 0 & 4  \\
   2 & 0 & 0  \\
\end{matrix} \right].$ Then $\left| {{A}^{*}} \right|=\left[ \begin{matrix}
   3 & 0 & 0  \\
   0 & 4 & 0  \\
   0 & 0 & 2  \\
\end{matrix} \right]  \;{\text{and}}\; \left| A \right|=\left[ \begin{matrix}
   2 & 0 & 0  \\
   0 & 3 & 0  \\
   0 & 0 & 4  \\
\end{matrix} \right].$
Consequently, 
$$\begin{aligned}
  & \underset{0\le t\le 1}{\mathop{\min }}\,\left\| \frac{\left( 1-t \right){{\left| A \right|}^{2}}+t{{\left| {{A}^{*}} \right|}^{2}}}{2}+\frac{{{\left| A \right|}^{4\left( 1-t \right)}}+{{\left| {{A}^{*}} \right|}^{4t}}}{4} \right\| \\ 
 & =\underset{0\le t\le 1}{\mathop{\min }}\,\left\| \left[ \begin{matrix}
   \frac{{{\left( 16 \right)}^{1-t}}+{{\left( 81 \right)}^{t}}}{4}+\frac{4+5t}{2} & 0 & 0  \\
   0 & \frac{{{\left( 81 \right)}^{1-t}}+{{\left( 256 \right)}^{t}}}{4}+\frac{9-7t}{2} & 0  \\
   0 & 0 & \frac{{{\left( 256 \right)}^{1-t}}+{{\left( 16 \right)}^{t}}}{4}+\frac{16-12t}{2}  \\
\end{matrix} \right] \right\| \\ 
 & =\underset{0\le t\le 1}{\mathop{\min }}\,\max \left\{ \frac{{{\left( 16 \right)}^{1-t}}+{{\left( 81 \right)}^{t}}}{4}+\frac{4+5t}{2},\frac{{{\left( 81 \right)}^{1-t}}+{{\left( 256 \right)}^{t}}}{4}+\frac{9-7t}{2},\frac{{{\left( 256 \right)}^{1-t}}+{{\left( 16 \right)}^{t}}}{4}+\frac{16-12t}{2} \right\} \\ 
 & =\underset{0\le t\le 1}{\mathop{\min }}\,\left\{ \begin{aligned}
  & \frac{{{\left( 256 \right)}^{1-t}}+{{\left( 16 \right)}^{t}}}{4}+\frac{16-12t}{2}\text{ for }0\le t\le 0.5286 \\ 
 & \frac{{{\left( 81 \right)}^{1-t}}+{{\left( 256 \right)}^{t}}}{4}+\frac{9-7t}{2}\text{ for }0.5287\le t\le 1 \\ 
\end{aligned} \right.\approx 9.32. \\ 
\end{aligned}$$
On the other hand,
	$$\frac{1}{2}\left\| {{\left| A \right|}^{}}+{{\left| {{A}^{*}} \right|}^{}} \right\|=3.5.$$
So, we have shown that, in this example,
	\[\underset{0\le t\le 1}{\mathop{\min }}\,\left\| \frac{\left( 1-t \right){{\left| A \right|}^{2}}+t{{\left| {{A}^{*}} \right|}^{2}}}{2}+\frac{{{\left| A \right|}^{4\left( 1-t \right)}}+{{\left| {{A}^{*}} \right|}^{4t}}}{4} \right\|^{\frac{1}{2}}<\frac{1}{2}\left\| {{\left| A \right|}^{}}+{{\left| {{A}^{*}} \right|}^{}} \right\|.\]
\end{example}

{\tiny $^{1}$Department of Mathematics, Islamic Azad University, Mashhad Branch, Mashhad, Iran}

{\tiny \textit{E-mail address:} shiva.sheybani95@gmail.com}

{\tiny $^{2}$Department of basic sciences, Princess Sumaya University for Technology, Amman 11941, Jordan}

{\tiny \textit{E-mail address:} sababheh@psut.edu.jo}

{\tiny $^{3}$Department of Mathematics, Payame Noor University (PNU), P.O. Box 19395-4697, Tehran, Iran}

{\tiny \textit{E-mail address:} hrmoradi@mshdiau.ac.ir}

%-----------------------------------------------------------------------------
%-----------------------------------------------------------------------------
\end{document}